\documentclass{article}
\usepackage{amsfonts}
\usepackage{amsmath}
\usepackage{hyperref}

\newtheorem{theorem}{Theorem}

\newtheorem{definition}[theorem]{Definition}

\newtheorem{lemma}[theorem]{Lemma}
\newtheorem{proposition}[theorem]{Proposition}

\newenvironment{proof}[1][Proof]{\textbf{#1.} }{\ \rule{0.5em}{0.5em}}
\newcommand{\func}{\operatorname}
\newcommand{\limfunc}{\operatorname}

\begin{document}

\title{Modular class of even symplectic manifolds}
\author{J. Monterde and J. Vallejo \\
Departament de Geometria i Topologia.\\
Universitat de Val\`{e}ncia (Spain).\\
{\small{e-mail: \texttt{juan.l.monterde@uv.es, jose.a.vallejo@uv.es}}}
}
\maketitle

\begin{abstract}
We provide an intrinsic description of the notion of modular class for an
even symplectic manifold and study its properties in this coordinate free
setting.
\end{abstract}

\section{Introduction}

The definition of the modular vector field of a Poisson manifold $%
(M;\{\_,\_\})$, is as follows: given a volume element $\eta$ on $M$, the
modular vector field $Z^{M}$ maps each function $f\in C^{\infty}(M)$ into
the divergence with respect to $\eta$ of the hamiltonian vector field
associated to $f$, i.e, 
\begin{equation}
Z^{M}(f):=\func{div}^{\eta}(X_{f})=\func{div}^{\eta }(\{\func{d}f,\_\}).
\label{eq0}
\end{equation}

What is called the modular class of $(M;\{\_,\_\})$, is its class in the
Poisson-Lichnerowicz cohomology (\cite{Lic 77}). The concept of modular
vector field was introduced by Koszul in \cite{Kos 84}, in his study of the
cohomology of a Poisson manifold, and Weinstein (in \cite{Wei 97}) has used
it as a tool to understand the modular automorphisms of von Neumann
algebras, observing that these share with their semiclassical limits
(Poisson algebras) the property of having modular automorphisms groups. The
concept has also appeared in geometry in the classification of quadratic
Poisson structures (see \cite{Duf-Har 91}). The modular vector field and the
related notion of volume element, has also been used intensively by O. M.
Khudaverdian and others in the study of graded Poincar\'{e}-Cartan
invariants, the geometry of Batalin-Vilkovisky formalism, etc (see \cite{Khu
81}, \cite{Khu 91}, \cite{Khu 98}). So we feel that an intrinsic,
geometrical study of these structures deserves attention.

The notion of modular class only needs a Poisson structure to be defined,
but we will center our attention in the non degenerate case.

In the graded setting, when a graded Poisson manifold $((M,\wedge \mathcal{E}%
),[\![\_,\_]\!])$ is given (see Sections $2$ and $3$ for the definitions), a
fundamental distinction appears: even though an appropriate definition of
divergence can be given, the analog of the mapping (\ref{eq0}) does not give
a derivation on ${\wedge }\mathcal{E}$ when the Poisson bracket is odd with
respect to the $\mathbb{Z}$-grading, but a generator for the Poisson
bracket, in the sense of Gerstenhaber algebras (see \cite{Khu 91}, \cite
{Kos-SCh 99} or \cite{Kos-Mon 01}). On the other hand, when the bracket is
even with respect to the $\mathbb{Z}$-grading the same mapping does give a
derivation on ${\wedge }\mathcal{E}$. So it is in this case that it makes
sense to develop the notions of graded modular vector field and modular
class.

In the nongraded case, it is a well known fact that any symplectic manifold $%
(M,\omega )$ is unimodular, i.e, it gives the zero class. Now suppose $M$ is
the base manifold of a given graded Poisson manifold $(M,\wedge \mathcal{E})$
whose graded Poisson bracket $[\![\_,\_]\!]$ is nondegenerate and extends
the Poisson bracket in $M$ defined by $\omega $. It is also known that this
bracket has an associated volume form which, in local coordinates, is
expressed by the Berezinian of the bracket matrix (see \cite{Ber 87}); using
this volume form, it can be seen that the modular vector field is zero, so $%
(M,\wedge \mathcal{E})$ is unimodular.

Our purpose in this paper, is to give a geometrical, coordinate free setting
for these results.We define the notion of symplectic Berezinian volume
element in an intrinsic way, and study how it changes with the section of
the Berezinian sheaf chosen, along with its relation to the canonical
Berezinian. As an application, we give a graded formulation of the
continuity equation of fluid mechanics.

\section{Graded forms on $(M,\Gamma(\Lambda E))$}

For the generalities on graded manifolds, see \cite{Kos 77}, \cite{Lei 80}
or \cite{Ber 87}; our approach here follows \cite{Mon-San 97}. Let $M$ be an 
$m$-dimensional smooth manifold, and let $C_{M}^{\infty }$ be the sheaf of
smooth functions on $M$. Let $E\rightarrow M$ be a vector bundle of rank $n$%
, and let $\mathcal{E}=\Gamma (E)$ be its sheaf of smooth sections. Let $%
\wedge \mathcal{E}=\Gamma (\Lambda E)$ be the sheaf of smooth sections of
the exterior algebra bundle $\wedge E\rightarrow M$.

We refer to \cite{Kos 77} or to \cite{Mon-San 97} for definitions of graded
vector field, graded differential form, insertion operator, $\iota(D)$ ($D$
being a graded vector field), exterior differential, $d^{G}$, and Lie
operator $\mathcal{L}_{D}^{G}$.

Being a graded homomorphism of graded modules, a graded differential form
has a degree. Thus, we can define a $\mathbb{Z}\times\mathbb{Z}-$bigrading
on the module of graded differential forms and we will say that a graded
differential form $\lambda$ has bidegree $(p,k)\in\mathbb{Z}\times\mathbb{Z}$
if 
\begin{equation*}
\lambda:\mathrm{\limfunc{Der}}\wedge\mathcal{E}\times.\overset{p)}{.}.\times%
\mathrm{\limfunc{Der}}\wedge\mathcal{E}\longrightarrow \wedge\mathcal{E}
\end{equation*}
and if, for all $D_{1},...,D_{p}\in$ $\mathrm{\limfunc{Der}}\wedge\mathcal{E}
$, 
\begin{equation*}
\left| \left\langle D_{1},...,D_{p};\lambda\right\rangle \right|
=\sum\limits_{i=1}^{p}\left| D_{i}\right| +k\text{.}
\end{equation*}

Using this bigrading, any graded $p-$differential form $\lambda$ can be
decomposed as a sum $\lambda=\lambda_{(0)}+...+\lambda_{(n)}$, where $%
\lambda_{(i)}$ is a homogeneous graded form of bidegree $(p,i)$.

A fundamental result is the following corollary to a theorem by Kostant (4.7
in \cite{Kos 77}).

\begin{proposition}
\label{cor-Ko} Every $d^{G}$-closed graded form of bidegree $(p,k)$ with $%
k>0 $ is exact.
\end{proposition}

Use will be made of the fact that the space of graded vector fields, $%
\mathrm{\limfunc{Der}}\wedge\mathcal{E}$, is a locally-free sheaf of ${\wedge%
}\mathcal{E}$-modules \cite{Kos 77}. See \cite{Mon 92} and \cite{Rot 90})
for an analysis ot its structure: Let $\mathcal{E}^{\ast}$ be the sheaf of
sections of the dual bundle $E^{\ast}\rightarrow M$. There is a monomorphism 
\begin{equation*}
i\colon\Gamma({\wedge}\mathcal{E})\otimes\mathcal{E}^{\ast}\hookrightarrow 
\mathrm{\limfunc{Der}}\wedge\mathcal{E}
\end{equation*}
On the other hand, let $\mathcal{X}(M)=\limfunc{Der}C_{M}^{\infty}$ be the
sheaf of smooth vector fields on $M$. A connection $\nabla$ on ${\wedge }%
\mathcal{E}$ gives, by definition, a morphism 
\begin{align*}
\Gamma({\wedge}\mathcal{E})\otimes\mathcal{X}(M) & \rightarrow \mathrm{%
\limfunc{Der}}\wedge\mathcal{E} \\
\alpha\otimes X & \mapsto\alpha\nabla_{X}.
\end{align*}

\section{Divergence operators and modular graded vector fields}

By definition, a \textit{divergence operator} on ${\wedge}\mathcal{E}$ is an
even linear map, $\func{div}:\mathrm{\limfunc{Der}}\wedge\mathcal{E}%
\rightarrow{\wedge}\mathcal{E}$, such that 
\begin{equation}
\func{div}(sD)=s\ \func{div}(D)+(-1)^{|s||D|}D(s)\ ,  \label{gooddiv}
\end{equation}
for any $D\in\mathrm{\limfunc{Der}}\wedge\mathcal{E}$ and any $s\in{\wedge}%
\mathcal{E}$.

The \textit{modular vector field} $Z^{M}$, associated to a divergence
operator $\limfunc{div}$ and a graded Poisson bracket $[\![\_,\_]\!]$ on $%
\wedge\mathcal{E}$, is the even graded vector field defined as 
\begin{equation}
s\in{\wedge}\mathcal{E}\mapsto D_{s}=[\![s,\_]\!]\in \mathrm{\limfunc{Der}}%
\wedge\mathcal{E}\mapsto\func{div}(D_{s})\in\wedge\mathcal{E}  \label{eq3.3}
\end{equation}

It is easy to check that when the even Poisson bracket is the Poisson
bracket associated to an even symplectic form, $\Theta$, then $Z^{M}$ is a
locally hamiltonian graded vector field. From now on, we shall work
exclusively in this case, this is, with an even symplectic form on $(M,\wedge%
\mathcal{E})$ and its associated even Poisson bracket $[\![\ ,\
]\!]_{\Theta} $.

\begin{lemma}
\label{ocho} Let $D = \sum\limits_{i\in\mathbb{N}}$ $D_{2i}\in \mathrm{%
\limfunc{Der}}\wedge\mathcal{E}$ be a locally hamiltonian even derivation.
Consider the decomposition (according to the $\mathbb{Z}-$degree) $%
\Theta=\Theta_{(0)}+\Theta_{(\geq2)}$. Then, $D$ is a graded hamiltonian
vector field for $\Theta$ if and only if $\iota_{D_{0}}\Theta_{(0)}$ is an
exact graded form.
\end{lemma}

\begin{proof}
It is an straightforward computation thanks to Prop. \ref{cor-Ko}.
\end{proof}

This means that the modular class just depends on the zero degree term of
the modular vector field.

\section{The symplectic Berezinian volume element and the modular class}

Let $\Theta$ be an even symplectic form on a graded manifold $(M,\wedge 
\mathcal{E})$ of dimension $(2n,m)$. We know that there are three objects
associated to the even symplectic form (see \cite{Rot 90}): an usual
symplectic form, $\omega$, on the base manifold $M$; a non degenerate
symmetric bilinear form, $g$, on $E^{\ast}$ and a connection, $\nabla$, on $%
E $, compatible with $g$, i.e, $\nabla g = 0$.

Let $\omega^{n}$ be the symplectic volume element on $M$, and let $\mu_{g}$
the metric volume element on $E$.

Given $s\in\wedge\mathcal{E}$ of compact support, we can define 
\begin{equation*}
\int_{\xi}s:=\int_{M}(i_{\mu_{g}}s)\omega^{n},
\end{equation*}
where $i_{\mu_{g}}s$ denotes the total contraction of $\mu_{g}\in
\Gamma(\Lambda^{m}E^{\ast})$ with $s$.

Such a definition includes, in an implicit way, the definition of a
Berezinian volume element, $\xi $. (See \cite{Lei 80}, \cite{HR-MM 85} or 
\cite{Kos-Mon 01})

We are going to define a divergence operator associated to the even
symplectic form through a Berezinian volume element, ${\xi}$. Given a
derivation $D\in\mathrm{\limfunc{Der}}\wedge\mathcal{E}$, there is a unique
section, denoted by $\func{div}^{\xi}(D)\in\wedge\mathcal{E}$ such that 
\begin{equation*}
-\int_{\xi}D(s)=\int_{\xi}\func{div}^{\xi}(D)\wedge s,
\end{equation*}
for all $s\in\wedge\mathcal{E}$ of compact support.

This is, indeed, a divergence operator.

\begin{proposition}
\begin{equation*}
\limfunc{div}\nolimits^{\xi }(s\wedge D)=s\wedge {\func{div}^{\xi }}%
(D)+(-1)^{|D||s|}D,
\end{equation*}
\end{proposition}

\begin{proof}
Just a matter of computation: 
\begin{align*}
\int_{\xi }\limfunc{div}\nolimits^{\xi }(s\wedge D)\wedge \overline{s}&
=-\int_{\xi }s\wedge D(\overline{s}) \\
& =-\int_{M}i_{\mu _{g}}(s\wedge D(\overline{s}))\omega ^{n} \\
& =-(-1)^{|D||s|}\int_{M}i_{\mu _{g}}(D(s\wedge \overline{s}))\omega
^{n}+(-1)^{|D||s|}\int_{M}i_{\mu _{g}}(D(s)\wedge \overline{s})\omega ^{n} \\
& =-(-1)^{|D||s|}\int_{\xi }D(s\wedge \overline{s})+(-1)^{|D||s|}\int_{\xi
}D(s)\wedge \overline{s} \\
& =(-1)^{|D||s|}\int_{\xi }\limfunc{div}\nolimits^{\xi }(D)\wedge s\wedge 
\overline{s}+(-1)^{|D||s|}\int_{\xi }D(s)\wedge \overline{s} \\
& =\int_{\xi }(s\wedge \limfunc{div}\nolimits^{\xi
}(D)+(-1)^{|D||s|}D(s))\wedge \overline{s}.
\end{align*}
\end{proof}

Now, we would like to know what happens when we change the section of the
Berezinian sheaf; for this, we recall that the Berezinian module is a right $%
\wedge \mathcal{E}$-module of rank $1$ (see \cite{HR-MM 85}). So, given a
Berezinian volume element $\xi $, any other Berezinian volume element is of
the kind $\xi .\bar{s}$ for an invertible even element, $\bar{s}\in \wedge 
\mathcal{E}$.

\begin{proposition}
If $\bar{s}$ is of compact support, then $\limfunc{div}\nolimits^{\xi \bar{s}%
}=\limfunc{div}\nolimits^{\xi }+\limfunc{d}\nolimits^{G}\log \bar{s}.$
\end{proposition}

\begin{proof}
From the definition of Berezinian, 
\begin{equation*}
\int\nolimits_{\xi \bar{s}}\_=\int\nolimits_{\xi }\bar{s}\wedge \_.
\end{equation*}
Now we have, for any $s\in \wedge \mathcal{E},$%
\begin{equation}
\int\nolimits_{\xi \bar{s}}D(s)=-\int\nolimits_{\xi \bar{s}}\limfunc{div}%
\nolimits^{\xi \bar{s}}(D)\wedge s=-\int\nolimits_{\xi }\bar{s}\wedge 
\limfunc{div}\nolimits^{\xi \bar{s}}(D)\wedge s.  \label{eq4.1}
\end{equation}
On the other hand, 
\begin{eqnarray}
\int\nolimits_{\xi \bar{s}}D(s) &=&\int\nolimits_{\xi }\bar{s}\wedge
D(s)=\int_{M}i_{\mu _{g}}(\overline{s}\wedge D(s))\omega ^{n}=  \label{eq4.2}
\\
&=&\int_{M}i_{\mu _{g}}(D(\overline{s}\wedge s))\omega ^{n}-\int_{M}i_{\mu
_{g}}(D(\overline{s})\wedge s)\omega ^{n}=  \notag \\
&=&\int\nolimits_{\xi }D(\bar{s}\wedge s)-\int\nolimits_{\xi }D(\bar{s}%
)\wedge s=  \notag \\
&=&-\int\nolimits_{\xi }\limfunc{div}\nolimits^{\xi }(D)\wedge \bar{s}\wedge
s-\int\nolimits_{\xi }\overline{s}\wedge \overline{s}^{-1}\wedge D(\bar{s}%
)\wedge s=  \notag \\
&=&-\int\nolimits_{\xi }\bar{s}\wedge \limfunc{div}\nolimits^{\xi }(D)\wedge
s-\int\nolimits_{\xi }\overline{s}\wedge \overline{s}^{-1}\wedge D(\bar{s}%
)\wedge s.  \notag
\end{eqnarray}
Equating (\ref{eq4.1}) and (\ref{eq4.2}), we obtain 
\begin{eqnarray*}
\limfunc{div}\nolimits^{\xi \bar{s}}(D) &=&\limfunc{div}\nolimits^{\xi }(D)+%
\overline{s}^{-1}\wedge D(\bar{s})= \\
&=&\limfunc{div}\nolimits^{\xi }(D)+D(\log \bar{s})= \\
&=&\limfunc{div}\nolimits^{\xi }(D)+\left\langle D;\limfunc{d}%
\nolimits^{G}\log \bar{s}\right\rangle ,
\end{eqnarray*}
and, from here, the statement.
\end{proof}

This enables us to give the following definition.

\begin{definition}
The \emph{modular class} of an even Poisson bracket is the class of any
modular vector field in the quotient $\mathrm{\mathrm{\limfunc{Der}}}\wedge 
\mathrm{\mathcal{E}}/\mathrm{\limfunc{Ham}}(\pi )$.
\end{definition}

Let us note how the notion of symplectic Berezinian is related to that of
canonical Berzinian. Given the volume form $\omega ^{n}$ on $M$ and the
metric volume $\mu _{g}$, as they are forms of maximal degree on $M$, there
must exist a function $f$ such that $\omega ^{n}=\limfunc{e}\nolimits^{f}\mu
_{g}$. If $s_{(\max )}$ denotes the maximal degree part of the section $s$,
also there must exist a $h$ with $s_{(\max )}=h\mu _{g}$, and we have that
the canonical Berzinian gives 
\begin{equation*}
\int\nolimits_{can}s=\int\nolimits_{M}s_{(\max )};
\end{equation*}
on the other hand, the symplectic Berezinian reads 
\begin{eqnarray*}
\int\nolimits_{symp}s &=&\int_{M}(i_{\mu _{g}}s)_{(\max )}\omega
^{n}=\int_{M}i_{\mu _{g}}(h\mu _{g})\limfunc{e}\nolimits^{f}\mu _{g} \\
&=&\int_{M}h\limfunc{e}\nolimits^{f}\mu _{g}=\int\nolimits_{M}\limfunc{e}%
\nolimits^{f}s_{(\max )},
\end{eqnarray*}
so $\limfunc{e}\nolimits^{f}$ is the section that passes from $%
\int\nolimits_{symp}$ to $\int\nolimits_{can}$. Then, from Proposition $4$,
the associated divergences are related through 
\begin{equation*}
\limfunc{div}\nolimits^{symp}=\limfunc{div}\nolimits^{can}+\limfunc{d}%
\nolimits^{G}f.
\end{equation*}

In the case of $(M,\omega ,g)$ a K\"{a}hler manifold, in which $f$ is a
constant function, $\limfunc{div}\nolimits^{symp}=\limfunc{div}%
\nolimits^{can}$.

The basic derivations in this setting are of the type $i_{\chi }$, for $\chi
\in \Gamma (E^{\ast })$, and $\nabla _{X}$, for a vector field $X$, and
where we can use the linear connection $\nabla $ induced by the even
symplectic form. Let us compute their divergences.

\begin{lemma}
Let $\nabla$ be a connection compatible with $g$, then, 
\begin{equation*}
\func{div}^{\xi}(i_{\chi})=0, \qquad\func{div}^{\xi}(\nabla_{X})=\func{div}%
^{\omega^{n}}(X).
\end{equation*}
\end{lemma}

\begin{proof}
Indeed, $i_{\chi}s$ is a section of degree $<m = \func{rk}(E)$, then $%
i_{\mu_{g}}i_{\chi}s=0$ for any $s$. For the other basic derivations, 
\begin{align*}
i_{\mu_{g}}(\nabla_{X}s)\omega^{n} &
=X(i_{\mu_{g}}s)\omega^{n}-(i_{\nabla_{X}\mu_{g}}s)\omega^{n}= \\
& =\mathcal{L}_{X}((i_{\mu_{g}}s)\omega^{n})-(i_{\mu_{g}}s)\mathcal{L}%
_{X}\omega^{n}= \\
& =\limfunc{d}i_{X}((i_{\mu_{g}}s)\omega^{n})+i_{X}\limfunc{d}%
((i_{\mu_{g}}s)\omega^{n})-(i_{\mu_{g}}s)\mathcal{L}_{X}\omega^{n}.
\end{align*}
Now, the first term $\limfunc{d}i_{X}((i_{\mu_{g}}s)\omega^{n})$ does not
contribute in the integral because it is an exact term. The second, $i_{X}%
\limfunc{d}((i_{\mu_{g}}s)\omega^{n})$, is equal to zero because $%
(i_{\mu_{g}}s)\omega^{n}$ is a top degree differential form on $M$. The
third term gives $(i_{\mu_{g}}s)\limfunc{div}\nolimits^{\omega^{n}}(X)$
because $\mathcal{L}_{X}\omega^{n}=\limfunc{div}\nolimits^{\omega^{n}}(X)%
\omega^{n}$. Finally, note that $\nabla_{X}\mu_{g}$ vanishes by hypothesis.

Therefore, $-\int\nolimits_{\xi}\nabla_{X}s=\int\nolimits_{\xi}\limfunc{div}%
\nolimits^{\omega^{n}}(X)s$.
\end{proof}

\begin{theorem}
Any even symplectic form on a graded manifold $(M,\wedge\mathcal{E})$ is
unimodular.
\end{theorem}

\begin{proof}
Recall Rothstein Theorem: $\Theta=\varphi^{\ast}(\Theta_{\omega,g,\nabla})$
for an automorphism $\varphi$ of $\wedge\mathcal{E}$ and where $\nabla$ is
compatible with $g$; thus, it is clear that if we prove that $\Theta
_{\omega,g,\nabla}$ is unimodular, $\Theta$ will also be. $\Theta
_{\omega,g,\nabla}$ is given by 
\begin{align*}
\left\langle \nabla_{X},\nabla_{Y};\Theta_{\omega,g,\nabla}\right\rangle &
=\omega(X,Y)+\frac{1}{2}R(X,Y,\_,\_) \\
\left\langle \nabla_{X},i_{\chi};\Theta_{\omega,g,\nabla}\right\rangle & =0
\\
\left\langle i_{\chi},i_{\psi};\Theta_{\omega,g,\nabla}\right\rangle &
=g(\chi,\psi),
\end{align*}
so that $(\Theta_{\omega,g,\nabla})_{(0)}$, which we shall denote $%
\Theta_{(0)}^{\omega,g,\nabla}$, is given by 
\begin{align*}
\left\langle \nabla_{X},\nabla_{Y};\Theta_{(0)}^{\omega,g,\nabla
}\right\rangle & =\omega(X,Y) \\
\left\langle \nabla_{X},i_{\chi};\Theta_{(0)}^{\omega,g,\nabla}\right\rangle
& =0=\left\langle i_{\chi},i_{\psi};\Theta_{(0)}^{\omega,g,\nabla
}\right\rangle .
\end{align*}

The graded Hamiltonian vector field associated to $f\in C^{\infty}(M)$
through the symplectic form $\Theta_{\omega,g,\nabla}$ is given by 
\begin{equation*}
D_{f} = \nabla_{X_{f}}+ h.d.t.
\end{equation*}

We have Lemma $2$, telling us that $D$ is a graded Hamiltonian vector field
if and only if $\iota _{D_{0}}\Theta _{(0)}$ is an exact graded form. On the
other hand, we know that 
\begin{equation}
\pi _{(0)}(Z^{M}(f))=\pi _{(0)}(\limfunc{div}(D_{f}))=\limfunc{div}(\nabla
_{X_{f}})=0.
\end{equation}
Therefore $Z^{M}=i_{N}+\text{\emph{higher degree terms,}}$ where $N\in 
\limfunc{End}\mathcal{E}$. But then, 
\begin{equation*}
\iota _{Z_{0}^{M}}\Theta _{(0)}^{\omega ,g,\nabla }=\iota _{i_{N}}\Theta
_{(0)}^{\omega ,g,\nabla }=0.
\end{equation*}
\end{proof}

\section{Applications}

In this section, we intend to provide some ideas about the possible
applications of these results. In the classical case, the notions of
divergence and vanishing modular class, are intimately related to
conservation laws along the flow of fluids; in fact, to one of the basic
equations of fluid dynamics, the continuity equation. We do not intend here
to give a complete description of the equations of graded fluids, we will
content ourselves with a study of what the graded continuity equation must
be (this is the only basic equation of fluid dynamics which is directly
related to the conservation of volume by the Hamiltonian flow).

Let us consider more concretely the classical situation we want to extend to
the graded case.

Let $V\in\mathcal{X}(M)$ be a vector field describing a classical dynamical
system (for instance, think of the velocities field on a fluid), and let $%
\{\varphi_{t}\}_{t\in\mathbb{R}}$ be its flow. Associated to any function $%
f\in C^{\infty}(M)$ (which describes the density of some observable on the
system), we have the continuity equation 
\begin{equation}
\frac{\partial f}{\partial t}+\limfunc{div}(fV)=0  \label{cont}
\end{equation}
(here we allow the possibility of a time dependence in $f$). This equation,
expresses the conservation of the total magnitude associated to $f$: 
\begin{equation}
\frac{d}{dt}\int_{M}f\mu=0,  \label{cons}
\end{equation}
where $\mu$ is a volume form on $M$, usually the symplectic volume form
coming from the hamiltonian structure of the dynamical system.

What would be the graded analog of (\ref{cont})?. We can not mimic the
physical reasoning of the classical case, because in the graded one there is
no notion of volume form (understood as a maximal degree graded form), but
we can extend the geometrical interpretation. For this, let us note that (%
\ref{cont}) can be rewritten as 
\begin{equation}
(\frac{\partial}{\partial t}+\mathcal{L}_{V})(f\mu)=0.  \label{contLie}
\end{equation}
The continuity equation in its form (\ref{contLie}), allows one to interpret 
$f\mu$ as a density form on the fluid which is dinamically conserved along
the flow $\{\varphi_{t}\}_{t\in\mathbb{R}}$. Here $f$ can be a volume
density, a charge density, etc. Moreover, this equation and its geometrical
interpretation carry over to graded manifolds. Now, an ``observable
density'' will be a superfunction $\rho\in\wedge\mathcal{E}$. A graded
vector field is a $D\in\mathrm{\limfunc{Der}}\wedge\mathcal{E}$ and its
flow, in general, is two-parameter dependent (see \cite{Mon-San 93} for
details on superflows), $\{\Phi_{(t,s)}^{\ast}\}_{(t,s)\in\mathbb{R}^{1|1}}$%
, where $\Phi :\mathbb{R}^{1|1}\times(M,\wedge\mathcal{E)}%
\rightarrow(M,\wedge\mathcal{E)}$. Thus, if $(t,s)$ are the (global)
supercoordinates of $\mathbb{R}^{1|1}$, the graded analog of (\ref{contLie})
would be the expression of the conservation of $\rho$ along the flow of $D$: 
\begin{equation}
(\frac{\partial}{\partial t}+\frac{\partial}{\partial s}+\mathcal{L}%
_{D}^{G})(\rho)=0,  \label{contLiegrad}
\end{equation}
where we have taken $\frac{\partial}{\partial t}+\frac{\partial}{\partial s}$
as the ``integrating model'' for supervector fields flows (see \cite{Mon-San
93}). Also, $\rho$ can eventually depend upon $s,t$.

By using our results (Proposition $3$ and Theorem $6$), we can recast (\ref
{contLiegrad}) in a form similar to the classical one (\ref{cont}): 
\begin{align*}
(\frac{\partial}{\partial t}+\frac{\partial}{\partial s}+\mathcal{L}%
_{D}^{G})(\rho) & =(\frac{\partial}{\partial t}+\frac{\partial}{\partial s}%
)(\rho)+D(\rho)= \\
& =(\frac{\partial}{\partial t}+\frac{\partial}{\partial s})(\rho
)+(-1)^{\left| D\right| \left| \rho\right| }(\limfunc{div}%
\nolimits^{\xi}(\rho D)-\rho\wedge\limfunc{div}\nolimits^{\xi}(D))= \\
& =(\frac{\partial}{\partial t}+\frac{\partial}{\partial s})(\rho
)+(-1)^{\left| D\right| \left| \rho\right| }(\limfunc{div}%
\nolimits^{\xi}(\rho D)).
\end{align*}
Thus, the equation of continuity reads now 
\begin{equation*}
(\frac{\partial}{\partial t}+\frac{\partial}{\partial s})(\rho)+(-1)^{\left|
D\right| \left| \rho\right| }(\limfunc{div}\nolimits^{\xi}(\rho D))=0.
\end{equation*}
Indeed, though it is not evident, this equation is of the ``conservation of
mass'' type. We only have to take into account the properties of the
superflows which are analogues to those of the classical flow of vector
fields. Let us denote by $(U,\wedge\mathcal{E}|_{U})$ an open superdomain
and by $\Phi_{(t,s)}^{\ast}(U,\wedge\mathcal{E}|_{U})$ the superdomain
obtained from the action of the superflow of $D$. Then, if $\int$ denotes
the berezinian integral, 
\begin{align*}
(\frac{\partial}{\partial t}+\frac{\partial}{\partial s})\int_{\Phi
_{(t,s)}^{\ast}(U,\wedge\mathcal{E}|_{U})}\rho & =\int_{(U,\wedge \mathcal{E}%
|_{U})}(\frac{\partial}{\partial t}+\frac{\partial}{\partial s}%
)\Phi_{(t,s)}^{\ast}\rho= \\
& =\int_{(U,\wedge\mathcal{E}|_{U})}\Phi_{(t,s)}^{\ast}\left[ (\frac
{\partial}{\partial t}+\frac{\partial}{\partial s})\rho+\mathcal{L}%
_{D}^{G}\rho\right] = \\
& =\int_{\Phi_{(t,s)}^{\ast}(U,\wedge\mathcal{E}|_{U})}\left[ (\frac
{\partial}{\partial t}+\frac{\partial}{\partial s})\rho+\mathcal{L}%
_{D}^{G}\rho\right] ,
\end{align*}
and the continuity equation is equivalent to 
\begin{equation}
(\frac{\partial}{\partial t}+\frac{\partial}{\partial s})\int\rho=0,
\label{conserv}
\end{equation}
which is a conservation equation.

Note how this result embodies the classical one about conservation of mass
in a fluid (for definiteness, moving on $\mathbb{R}^{2}$ with its usual
symplectic and metric structure). It suffices to take $\rho(\overrightarrow
{x},t)=f(\overrightarrow{x},t)\mu$ (where $f$ is the density of the fluid
and $\mu$ is the symplectic volume form on $\mathbb{R}^{2})$, and $D=%
\mathcal{L}_{X}$ (where $X$ is the field of velocities) as a derivation on $(%
\mathbb{R}^{2},\Gamma(\Lambda T^{\ast}\mathbb{R}^{2}))$, and then, by the
definition of berezinian integral, (\ref{conserv}) leads to 
\begin{align*}
0 & =(\frac{\partial}{\partial t}+\frac{\partial}{\partial s})\rho +\mathcal{%
L}_{D}^{G}\rho= \\
& =\frac{\partial}{\partial t}(f\mu)+\mathcal{L}_{X}(f\mu)= \\
& =\frac{\partial f}{\partial t}\mu+\limfunc{div}(fX)\mu,
\end{align*}
that is, the classical equation 
\begin{equation*}
\frac{\partial f}{\partial t}+\limfunc{div}(fX)=0.
\end{equation*}

The advantage of the equation (\ref{conserv}), is that it allows to consider
all kinds of magnitudes expressibles as differential forms, in the spirit of
the generalization of classical mechanics proposed by Michor (see \cite{Mic
85}).

\textbf{Acknowlegements}. A previous version of this result was presented in
Colloque IHP, ``G\'{e}ometrie diff\'{e}rentielle et physique
math\'{e}matique'', Paris, June 2001. We thank Y. Kosmann-Schwarzbach, A.
Weinstein and G. Tuynman for suggesting us the possibility of a simpler
proof.

This work has been partially supported by the Spanish Ministerio de
Educaci\'{o}n y Cultura, Grant PB-97-1386.

\end{document}